\documentclass{amsart}
\usepackage{amsfonts,latexsym,graphicx,amssymb,amsmath,url,amsthm}
\newtheorem{theorem}{Theorem}
\newtheorem{lemma}{Lemma}
\newtheorem{corollary}{Corollary}

\theoremstyle{definition}
\theoremstyle{remark}
\newtheorem*{remark}{{\bf{Remark}}}
\newcommand{\N}{\mathbb N}
\newcommand{\Z}{\mathbb Z}

\begin{document}

\title{Longer gaps between values of binary quadratic forms}
\author[R. Dietmann] {Rainer Dietmann}
\address{Department of Mathematics, Royal Holloway,
University of London, Egham, Surrey, TW20 0EX, United Kingdom}
\email{rainer.dietmann@rhul.ac.uk} 
\author[C. Elsholtz]{Christian Elsholtz}
\address{Institute of Analysis and Number Theory,
Graz University of Technology,
Kopernikusgasse 24/II, Graz,
A-8010 Graz, Austria}
\email{elsholtz@math.tugraz.at}
\begin{abstract}
We prove new lower bounds on large gaps between integers which are sums of 
two squares, or are represented by {\emph{any}}
 binary quadratic form of discriminant $D$,
improving results of Richards.
Let $s_1, s_2, \ldots$ be the sequence of positive integers,
arranged in increasing order, that are
representable by \emph{any} binary quadratic form of fixed
discriminant $D$, then 
\[
  \limsup_{n \rightarrow \infty} \frac{s_{n+1}-s_n}{\log s_n}
  \ge \frac{\varphi(|D|)}{2|D|(1+\log \varphi(|D|))}\gg \frac{1}{\log \log |D|},
\]
improving a lower bound of $\frac{1}{|D|}$ of Richards.
In the special case of sums of two squares, we improve Richards's bound of
$1/4$ to $\frac{195}{449}=0.434\ldots$.
We also generalize Richards's result in another direction:
Let $d=2^r d'$ where $d'$ is odd.
For all $k\in \N$ there exists a smallest positive integer $y_k$ such that
none of the integers $y_k+j^d, 1\leq j \leq k$, is a sum of two squares.
Moreover,
\[
  \limsup_{k \rightarrow \infty}
  \frac{k}{\log y_k} \ge \frac{1}{4d'}.
\]
\end{abstract}
\maketitle
\section{Introduction}
Let
\[
   (s_1, s_2, s_3, s_4, s_5, s_6 \ldots)=(0,1,2,4,5,8,\ldots)
\]
denote the sequence 
of integers, in increasing order,
which can be written as a sum of two squares of integers.
The question of the size of large gaps between these integers 
was investigated by Tur\'{a}n and Erd\H{o}s \cite {Erdos:1951},
Warlimont \cite{Warlimont:1976} and Richards \cite{Richards:1982};
see also \cite{hooley}, \cite{bw}, \cite{maynard} and \cite{bc} for
related or more recent work by Hooley, Balog and Wooley, Maynard,
and Bonfoh and Enyi.
Erd\H{o}s writes that Tur\'{a}n proved that infinitely often
\[
  s_{n+1}-s_n \gg \frac{\log s_n}{\log \log s_n}
\]
holds, which Erd\H{o}s \cite{Erdos:1951} improved to
\begin{equation}
\label{bungalow}
  s_{n+1}-s_n \gg \frac{\log s_n}{\sqrt{\log \log s_n}}.
\end{equation}
In fact, Erd\H{o}s's result was a bit more general, and
Warlimont \cite{Warlimont:1976}
independently obtained the same estimate \eqref{bungalow}, again
in a more general context
for sequences with hypotheses slightly different from \cite{Erdos:1951}.
In a very short and elegant paper
Richards \cite{Richards:1982} improved this further to
\begin{equation}
\label{besichtigung}
  \limsup_{n \rightarrow \infty} \frac{s_{n+1}-s_n}{\log s_n}
  \ge \frac{1}{4}.  
\end{equation}
In fact, the result he obtained was again more general:
Fix a fundamental discriminant $D$ and denote by $s_1, s_2, \ldots$
the integers, in increasing order, representable by
\emph{any} binary quadratic form of discriminant $D$. Then
\begin{equation}
\label{norden}
  \limsup_{n \rightarrow \infty} \frac{s_{n+1}-s_n}{\log s_n}
  \ge \frac{1}{|D|}.  
\end{equation}
The special case $D=-4$ corresponds to sums of two squares and
recovers \eqref{besichtigung}. 

Apparently, Richards's record has not been broken since 1982.
In this paper we obtain the
following improvements to \eqref{besichtigung} and \eqref{norden}.
\begin{theorem}
\label{staines}
Let $s_1<s_2<\ldots$ be the sequence of positive integers that are sums
of two squares. Then
\[
  \limsup_{n \rightarrow \infty} \frac{s_{n+1}-s_n}{\log s_n}
  \ge \frac{195}{449} = 0.434\ldots
\]
\end{theorem}
\begin{theorem}
\label{thames}
Let $D$ be a fundamental discriminant, i.e. $D \equiv 1 \pmod 4$ and
$D$ being squarefree, or $D \equiv 0 \pmod 4$, $\frac{D}{4}$ being
squarefree and $\frac{D}{4} \equiv 2 \pmod 4$ or $\frac{D}{4} \equiv
3 \pmod 4$. Further, let $(s_1, s_2, \ldots)$ be the sequence of positive
integers, in increasing order, that are
representable by \emph{any} binary quadratic form of discriminant
$D$. Then
\[
  \limsup_{n \rightarrow \infty} \frac{s_{n+1}-s_n}{\log s_n}
  \ge \frac{\varphi(|D|)}{2|D|(1+\log \varphi(|D|))},
\]
where $\varphi$ denotes Euler's totient function.
\end{theorem}
Note that
\[
  \frac{\varphi(|D|)}{|D|} \gg \frac{1}{\log \log |D|},
\]
implying that
\[
  \limsup_{n \rightarrow \infty} \frac{s_{n+1}-s_n}{\log s_n}
  \gg \frac{1}{\log |D| \log \log |D|}.
\]
This is a significant improvement over the bound $1/D$ by Richards
and now the dependence on $|D|$
has become very mild.

Our approach largely follows that of Richards, 
but has a much more careful analysis
which prime factors of medium size are required for the argument.

Let us include Richards's proof of \eqref{besichtigung} first,
as its simplicity helps
with understanding. Fix any $\varepsilon>0$ and then
fix a sufficiently large gap size $k$.
For each prime $p \leq 4k, p \equiv
3 \pmod 4$ let $\beta=\beta(p)$ be the highest power with
$p^{\beta}\leq 4k$. Let
\begin{equation}
\label{product}
  P=\prod_{p \leq 4k,\, p \equiv 3 \pmod 4} p^{\beta(p)+1}.
\end{equation}
Define $y \in \{1, \ldots, P\}$ by $4y\equiv -1 \pmod P$.
Then as shown below, none of the integers in
\[
  I=\{y+1, \ldots ,y+k\}
\]
is the sum of two squares.
By the prime number theorem in arithmetic progressions
(see for example formula (17.1) in \cite{IK}), as $k$ and thus
$P$ is sufficiently large in terms of $\varepsilon$, we have
\[
  P<(4k)^{2 (1+\varepsilon)(2k/\log 4k)}=\exp ((1+\varepsilon)4k),
\]
whence 
\[
  k>\frac{\log P}{4(1+\varepsilon)}.
\]
Since $4y \equiv -1 \pmod P$, we have that
\[4(y+j)\equiv 4j-1 \pmod P, \text{ for } 1\leq j \leq k.\]
Now $4j-1$ is exactly divisible by some prime power
$p^{\alpha}$ with $p \equiv 3 \pmod 4$ and $\alpha$ odd.
As $\alpha \leq \beta(p)$, and $P$ is divisible by $p^{\beta+1}$, it follows
that $y+j$ is also exactly divisible by $p^{\alpha}$, and therefore $y+j$ is
indeed not the sum of two squares.
This proves \eqref{besichtigung}. In fact,
the argument generalizes to yield the following result on long gaps
between sums of two squares in sparse sequences.
\begin{theorem}
\label{shift}
Let $d=2^r d'$ where $d'$ is odd.
For all $k\in \N$ there exists a smallest positive integer $y_k$ such that
none of the integers $y_k+j^d, 1\leq j \leq k$, is a sum of two squares.
Moreover,
\[
  \limsup_{k \rightarrow \infty}
  \frac{k}{\log y_k} \ge \frac{1}{4d'}.
\]
\end{theorem}
\begin{remark}
When $d=2^r$, i.e. $d'=1$, the quantitative bound is as good as
Richards'; in particular, the special
case $d=1$ recovers exactly Richards's result \eqref{besichtigung}.
\end{remark}

As the details of the proofs of Theorems \ref{staines} and \ref{thames}
are quite technical, we outline the underlying ideas and structure of 
the proofs.
An even more elaborate
analysis along the same lines
in case of Theorem \ref{thames} could probably lead
to a further slight improvement; in case of Theorem \ref{shift}
we refrained from introducing our new refinement in order to keep the
statement clean as in this case 
our focus was more on the qualitative side of the
result.
The key observation for an improvement in the argument above is
as follows:
We concluded that $4j-1$ is exactly divisible by some prime power
$p^\alpha$ with $p \equiv 3 \pmod 4$ and odd $\alpha$,
say $4j-1=p^\alpha r$. If $4j-1<4k$ is
composite, then $p<\frac{4}{5}k$. Primes $p$ with
$\frac{4}{5}k \le p \le 4k$ included in the product \eqref{product}
therefore are only used once in the argument, namely when $p=4j-1$.
Now integers in $I=\{y+1, \ldots, y+k\}$ which are congruent to
$3\pmod 4$ are obviously not sums of two squares.
Hence, additionally assuming that $y\equiv 0 \pmod 4$, we conclude that for 
$j\equiv 3\pmod 4$ we trivially know that $y+j$ is not a sum of two squares.
As $4j-1 \equiv 11 \pmod {16}$ in this case, we deduce that
primes $p\equiv 11\pmod {16}$, with $4k/5 \le p \le 4k$, are not needed
in the product \eqref{product}.
In other words, 
for fixed $k$ one can use a smaller $P$, or for a given size of $y$,
one can find a larger $k$ than in Richards's argument. This basic
approach just described in fact would replace the $\frac{1}{4}$
in \eqref{besichtigung} by
\[
  \left(2 \times \frac{1}{2} \left(
  \frac{4}{5}+\frac{3}{4} \left(4-\frac{4}{5}\right) \right) \right)^{-1}
  = \frac{5}{16} = 0.3125,
\]
but it can be further refined in two ways:
First, considering higher powers of $2$ one finds a larger proportion of
residue classes for $j$ one can dispense with; for example
the residue classes $j\equiv 3,6,7$ modulo $8$ immediately rule out
that $y+j$ is a sum of two squares, provided that $y \equiv 0 \pmod 8$.
Therefore the primes $p=4j-1 \ge \frac{4}{5}k$
with $p \equiv 11$, $23$, or $27 \pmod {32}$
are not needed. Secondly, also smaller primes $p$ can be considered;
for example primes $p$ with $\frac{4}{9}k \le p < \frac{4}{5}k$ can
only occur in the argument if either $4j-1=p$ or $4j-1=5p$. One
residue class one can disregard here is for example $p \equiv 11
\pmod {32}$, as for such $p$ both $p$ and $5p$ are modulo $32$ in
the set of residue classes $11$, $23$ and $27$ which were ruled out above
for $4j-1$. Implementing and optimizing these kinds of
improvements lead to the result in Theorem \ref{staines}.

A similar idea is used to prove Theorem \ref{thames}. Here instead of
$4j-1$ the progression $|D|j+r$ shows up, for some $r$ with
$(\frac{D}{r})=-1$. Whereas Richards \cite{Richards:1982} uses all
primes $p$ with $(\frac{D}{p})=-1$, it turns out that for large
primes $p$ used only once, namely when $|D|j+r=p$
(i.e. $L/\ell_2<p\le L$), only $p$ with
$p \equiv r \pmod {|D|}$ need to be considered. Similarly, for
somewhat smaller primes $p$ used only twice, only two residue
classes modulo $|D|$ have to be covered, and so on. This leads to
a considerably smaller product $P$.

\smallskip
\emph{Notation:} We say that a prime power $p^\alpha$
\emph{exactly divides} an integer $n$ if $p^\alpha$ divides $n$ but
$p^{\alpha+1}$ does not. The symbol $\left( \frac{m}{n} \right)$
always denotes the Kronecker symbol of $m$ over $n$.
Also,
as the proof of Theorem \ref{staines} in section \ref{regen} 
requires quite heavy notation, we have a summary
of abbreviations in appendix \ref{smalltalk} and some examples for the
sets of residue classes introduced in appendix \ref{english}.

\smallskip
\emph{Acknowledgments:} Large parts of this paper were written during
mutual visits of the authors, who would like to express thanks to
the TU Graz, Royal Holloway, and the ETH Z\"urich for support and
favourable working conditions. We also want to thank Igor Shparlinski
for helpful comments regarding the exposition of this paper.

\section{Longer gaps between sums of two squares}
\label{regen}
In this section we will prove Theorem \ref{staines}.
For all $\ell \in \N$ with $\ell \ge 2$ define
\begin{equation}
\label{aquaracer}
  S_\ell = \{2^ab \in \{1, \ldots, 2^\ell\}:
  a \le \ell-2, b \equiv 3 \pmod 4\}.
\end{equation}
Clearly
\begin{equation}
\label{tag}
  S_{\ell} \subset S_{\ell+1}
\end{equation}
for all $\ell \ge 2$.
In the following it is convenient to use the projection
\begin{align*}
  \pi_\ell : \; & \Z \rightarrow \{1, \ldots, 2^\ell\}\\
  & x \mapsto n \in \{1, \ldots, 2^\ell\} \; \text{such that} \;
  x \equiv n \pmod {2^\ell}.
\end{align*}
\begin{lemma}
\label{lem1}
For $\ell \ge 2$ we have
\[
  \#S_\ell = 2^{\ell-1}-1.
\]
Moreover, if $\pi_\ell(n) \in S_\ell$, then $n$ is not the sum of two
squares.
\end{lemma}
\begin{proof}
Suppose that $c \in \{1, \ldots, 2^\ell\}$ is of the form
$c=2^ab$ for $a \le \ell-2$ and
$b \equiv 3 \pmod 4$. If $n \in \N$
satisfies $n \equiv c \pmod {2^\ell}$, then
$n=2^\ell m+2^ab=2^a(2^{\ell-a}m+b)$ for a suitable integer $m$. Now
$2^{\ell-a}m+b \equiv 3 \pmod 4$, so $n$ must be divisible by a prime $p$
with $p \equiv 3 \pmod 4$ to an exact odd power, whence $n$ cannot be
a sum of two squares. It is immediate to check that the number of such $c$
of the form above is
\[
  \sum_{a=0}^{\ell-2} 2^{\ell-a-2} = 2^{\ell-1}-1,
\]
because for fixed $a \le \ell-2$ there are exactly
$2^{\ell-a-2}$ elements of $\{1, \ldots, 2^{\ell-a}\}$ that are
congruent to $3$ modulo $4$.
\end{proof}
Define the map $\tau$ by
\[
  \tau : \Z \rightarrow \Z; \quad j \mapsto 4j-1,
\]
and for $\ell \in \N$, $\ell \ge 2$ define
\[
  T_{\ell+2} = \tau(S_\ell) = \pi_{\ell+2}(\tau(S_\ell))
  \subset \{1, \ldots, 2^{\ell+2}\}.
\]
Again, one observes that
\begin{equation}
\label{27minuten}
  T_\ell \subset T_{\ell+1}
\end{equation}
for all $\ell \ge 4$.
\begin{lemma}
\label{20bahnen}
Let $\ell \ge 2$ and
$s \in S_{\ell}$ with $s \ne 3 \times 2^{\ell-2}$.
Then we have $\pi_\ell(s+2^{\ell-1}) \in S_\ell$.
\end{lemma}
\begin{proof}
As $s \in S_\ell$, $s \ne 3 \times 2^{\ell-2}$, it follows that
$s \equiv 2^a b \pmod {2^\ell}$
where $a \le \ell - 3$ and $b \equiv 3 \pmod 4$.
Hence $s+2^{\ell-1} \equiv 2^a(b+2^{\ell-1-a}) \pmod {2^\ell}$,
where $b+2^{\ell-1-a} \equiv b \equiv 3 \pmod 4$, whence
$\pi_\ell(s+2^{\ell-1}) \in S_\ell$. 
\end{proof}
\begin{corollary}
\label{cor1}
Let $\ell \ge 2$ and $t \in T_{\ell+2}$ with $t \ne 3 \times 2^{\ell}-1$.
Then $\pi_{\ell+2}(t+2^{\ell+1}) \in T_{\ell+2}$.
\end{corollary}
\begin{lemma}
\label{pilz}
Let $\ell \ge 3$. Then $3 \times 2^\ell-1 \not \in T_{\ell+1}$.
\end{lemma}
\begin{proof}
Let $x=3 \times 2^\ell-1$.
Since $T_{\ell+1} \subset \{1, \ldots, 2^{\ell+1}\}$ and
$x>2^{\ell+1}$, clearly $x \not \in T_{\ell+1}$.
\end{proof}
Let
\begin{equation}
\label{heuer}
  U_3 = \{3\} \subset S_3,
\end{equation}
and for $\ell \ge 4$, recursively define
\begin{equation}
\label{rek}
  U_\ell = U_{\ell-1} \cup \{u+2^{\ell-1}: u \in U_{\ell-1}\}
  \cup \{3 \times 2^{\ell-2}\}.
\end{equation}
It follows by induction from \eqref{aquaracer}, \eqref{tag},
Lemma \ref{20bahnen}, \eqref{heuer} and \eqref{rek}
that $U_\ell \subset S_\ell$ for all $\ell \ge 3$.
Moreover, for $\ell \ge 2$ define
\[
  V_\ell = \{s \in S_\ell: \pi_{\ell+2}(5\tau(s)) \in T_{\ell+2}\}.
\]
The important property of $V_\ell$ used later is that if
$x \in T_\ell$ is also in the subset $\tau(V_\ell)$ of $T_\ell$, then
$\pi_\ell(5x) \in T_\ell$ as well.

In a similar way, define
\[
  W_5=\{24\} \subset U_5,
\]
and, for $\ell \ge 6$,
\[
  W_\ell = W_{\ell-1} \cup \{u+2^{\ell-1}: u \in W_{\ell-1}\}
  \cup \{3 \times 2^{\ell-2}\}.
\]
As above one shows that
\[
  W_\ell \subset U_\ell
\]
for all $\ell \ge 5$. Further, for $\ell \ge 2$ let
\[
  R_\ell = \{s \in S_\ell: \pi_{\ell+2}(5\tau(s)) \in T_{\ell+2}
  \text{ and } \pi_{\ell+2}(9\tau(s)) \in T_{\ell+2}\}.
\]
Note that
\[
  R_\ell \subset V_\ell
\]
for all $\ell \ge 5$, and that the important property of $R_\ell$ used
later is that if
$x \in T_\ell$ is also in the subset $\tau(R_\ell)$ of $T_\ell$, then
$\pi_\ell(5x) \in T_\ell$ and $\pi_\ell(9x) \in T_\ell$ as well.

\begin{lemma}
\label{hoch_pokern}
Let $\ell \ge 3$. Then $\#U_\ell = 2^{\ell-2}-1$ and $U_\ell \subset V_\ell$.
\end{lemma}
\begin{proof}
We prove the lemma by induction on $\ell$. For $\ell=3$, it is
immediate to check that
\[
  U_3 = V_3 = \{3\},
\]
whence
\[
  \#U_3 = 1 = 2^{\ell-2}-1,
\]
and no element in $U_3$ is divisible by $2^{\ell-1}$.
Now let $\ell \ge 4$, and suppose that $\#U_{\ell-1} = 2^{\ell-3}-1$,
no element in $U_{\ell-1}$ is divisible by $2^{\ell-2}$ and
$U_{\ell-1} \subset V_{\ell-1}$. The three sets on the right hand
side of \eqref{rek} are disjoint as $U_{\ell-1} \subset
\{1, \ldots, 2^{\ell-1}\}$, $\{u+2^{\ell-1}: u \in U_{\ell-1}\}
\subset \{2^{\ell-1}+1, \ldots, 2^{\ell}\}$, and no element
in $U_{\ell-1}$ is divisible by $2^{\ell-2}$. Hence
\[
  \#U_\ell = 2\#U_{\ell-1} + 1 = 2^{\ell-2}-1
\]
and no element in $U_\ell$ is divisible by $2^{\ell-1}$.
To prove $U_\ell \subset V_\ell$,
let $s \in U_\ell$.

\smallskip
\noindent
\textbf{Case I:} $s=3 \times 2^{\ell-2}$.
Then
\begin{align*}
  \frac{1}{4} \left( 5\tau(s)-3 \times 2^{\ell+2} + 1 \right)
  &= \frac{1}{4} \left( 5 \times (4 \times 3 \times 2^{\ell-2}-1)
  - 3 \times 2^{\ell+2} + 1 \right) \\ &= 3 \times 2^{\ell-2}-1,
\end{align*}
hence
\[
  5\tau(s)\equiv 4 \times (3 \times 2^{\ell-2}-1)-1
  \pmod {2^{\ell+2}}.
\]
Now $3 \times 2^{\ell-2} - 1 \equiv 3 \pmod 4$, so
$3 \times 2^{\ell-2}-1 \in S_\ell$, whence
$4 \times (3 \times 2^{\ell-2}-1)-1 \in T_{\ell+2}$, so
$\pi_{\ell+2}(5\tau(s)) \in T_{\ell+2}$.

\smallskip
\noindent
\textbf{Case II:} $s \ne 3 \times 2^{\ell-2}$.
Then by definition of $U_\ell$, we have $\pi_{\ell-1}(s) \in
U_{\ell-1}$, so $\pi_{\ell+1}(5\tau(s)) \in T_{\ell+1}$ by our
inductive assumption $U_{\ell-1} \subset V_{\ell-1}$.
Hence $\pi_{\ell+2}(5\tau(s)) \in T_{\ell+1}$ or
$\pi_{\ell+2}(5\tau(s)) \in \{u+2^{\ell+1}:u \in T_{\ell+1}\}$.
If $\pi_{\ell+2}(5\tau(s)) \in T_{\ell+1}$ then by \eqref{27minuten}
we immediately obtain
$\pi_{\ell+2}(5\tau(s)) \in T_{\ell+2}$ as required, whereas
if $\pi_{\ell+2}(5\tau(s)) \in \{u+2^{\ell+1}:u \in T_{\ell+1}\}$
then \eqref{27minuten}, Lemma \ref{pilz} and
Corollary \ref{cor1} again
yield $\pi_{\ell+2}(5\tau(s)) \in T_{\ell+2}$.
\end{proof}
\begin{lemma}
\label{qd}
Let $\ell \ge 5$. Then $\#W_\ell = 2^{\ell-4}-1$ and $W_\ell \subset R_\ell$.
\end{lemma}
\begin{proof}
We use a similar strategy as in the proof of Lemma \ref{hoch_pokern}; the
proof for $\#W_\ell = 2^{\ell-4}-1$ is completely analogous, so let us focus
on the second part $W_\ell \subset R_\ell$.
The case $\ell=5$ is immediately checked directly. Now suppose that
$\ell \ge 6$ and
$W_{\ell-1} \subset R_{\ell-1}$. For $s \ne 3 \times
2^{\ell-2}$ we argue in exactly the same way as in the proof of Lemma
\ref{hoch_pokern}. Therefore let us only discuss the case
$s=3 \times 2^{\ell-2}$. Here
\begin{align*}
  \frac{1}{4} \left( 9\tau(s)-6 \times 2^{\ell+2} + 1 \right) 
  &= \frac{1}{4} \left(9 \times (4 \times 3 \times 2^{\ell-2}-1)-
  6 \times 2^{\ell+2} + 1 \right)\\
  &= 2 \times (3 \times 2^{\ell-3}-1),\\
\end{align*}
hence
\[
  9\tau(s) \equiv 4 \times 2 \times (3 \times 2^{\ell-3}-1)-1
  \pmod {2^{\ell+2}}.
\]
Now $2 \times (3 \times 2^{\ell-3}-1) \equiv 6 \pmod 8$, so
$2 \times (3 \times 2^{\ell-3}-1) \in S_\ell$, thus
$4 \times 2 \times (3 \times 2^{\ell-3}-1)-1 \in T_{\ell+2}$ and
$\pi_{\ell+2}(9\tau(s)) \in T_{\ell+2}$; from the proof of
Lemma \ref{hoch_pokern} we already know that
$\pi_{\ell+2}(5\tau(s)) \in T_{\ell+2}$.
\end{proof}
We now follow the idea of Richards \cite{Richards:1982} already
explained in the introduction.
Let $\varepsilon>0$. Then, in terms of $\varepsilon$,
fix a sufficiently large positive integer $\ell \ge 5$ and a
sufficiently large positive integer $k$, and let
the sets $S_\ell$, $T_{\ell+2}$, $U_\ell$, $V_\ell$, $W_\ell$, $R_\ell$
be defined as above. In the following it is convenient to define
\[
  A:=T_{\ell+2}, \quad B:=\pi_{\ell+2}(\tau(U_\ell)), \quad
  C:=\pi_{\ell+2}(\tau(W_\ell)).
\]
Note that
\[
  C \subset B \subset A
\]
since $W_\ell \subset U_\ell \subset S_\ell$.
By Lemma \ref{lem1}, Lemma \ref{hoch_pokern} and Lemma \ref{qd} we have
\begin{align*}
  \#A & =\#T_{\ell+2} =\#S_\ell=2^{\ell-1}-1, \quad
  \#B=\#U_\ell=2^{\ell-2}-1,\\
  \#C & =\#W_\ell=2^{\ell-4}-1.
\end{align*}
Hence if $\ell$ is chosen sufficiently large in terms of $\varepsilon$, then
\begin{equation}
\label{eq1}
  \frac{\#A}{\varphi(2^{\ell+2})} \ge \frac{1}{4}(1-\varepsilon); \quad
  \frac{\#B}{\varphi(2^{\ell+2})} \ge \frac{1}{8}(1-\varepsilon); \quad
  \frac{\#C}{\varphi(2^{\ell+2})} \ge \frac{1}{32}(1-\varepsilon). \quad
\end{equation}
Now for each prime $p \le 4k$ define
\[
  \beta(p) = \max_{p^m \le 4k} m,
\]
let
\[
  X = (1+\varepsilon) \frac{4}{13} k, \quad
  Y = (1+\varepsilon) \frac{4}{9} k, \quad
  Z = (1+\varepsilon) \frac{4}{5} k, \quad
\]
and let
\begin{align*}
  P = & 2^{\ell}
      \prod_{\substack{p_1 \le X:\\p_1 \equiv 3 \pmod 4}} p_1^{\beta(p)+1}
      \prod_{\substack{X<p_2 \le Y:\\ p_2 \equiv 3 \pmod 4,\\
      \pi_{\ell+2}(p_2) \not \in C}} p_2^{\beta(p)+1} \\
    & \times \prod_{\substack{Y<p_3 \le Z:\\ p_3 \equiv 3 \pmod 4,\\
      \pi_{\ell+2}(p_3) \not \in B}} p_3^{\beta(p)+1}
      \prod_{\substack{Z<p_4 \le 4k:\\ p_4 \equiv 3 \pmod 4,\\
      \pi_{\ell+2}(p_4) \not \in A}} p_4^{\beta(p)+1},
\end{align*}
where $p_1, \ldots, p_4$ denote prime numbers.
Then by the prime number theorem in arithmetic progressions
(see for example formula (17.1) in \cite{IK}), using
the upper bound $p^{\beta(p)+1} \le (4k)^2$, the lower bounds
\eqref{eq1} and the fact that all elements in $A$ are congruent to
$3$ modulo $4$, we obtain
\[
  P \le (4k)^{2(1+\varepsilon)k\alpha/\log(4k)},
\]
where
\begin{align*}
  \alpha & = \frac{1}{2} \left(
  \frac{4}{13} + \left( \frac{4}{9}-\frac{4}{13} \right) \frac{15}{16} +
  \left( \frac{4}{5}-\frac{4}{9} \right) \frac{3}{4} +
  \left(4-\frac{4}{5} \right) \frac{1}{2} \right)\\
  &= \frac{1}{2} \times \frac{449}{195}.
\end{align*}
Hence
\[
  P \le \exp \left((1+\varepsilon) \frac{449}{195}k\right),
\]
so
\begin{equation}
\label{lowball}
  \frac{k}{\log P} \ge \frac{195}{(1+\varepsilon)449}.
\end{equation}
Now use the Chinese Remainder Theorem to find $y \in \{1, \ldots, P\}$
such that
\begin{itemize}
\item[(i)] $2^{\ell} \mid y$,
\item[(ii)] if $p \equiv 3 \pmod 4$ and $p \le X$, then
            $4y \equiv -1 \pmod {p^{\beta(p)+1}}$,
\item[(iii)] if $p \equiv 3 \pmod 4$, $X<p \le Y$ and
             $\pi_{\ell+2}(p) \not \in C$,\\ then $4y \equiv -1 \pmod
             {p^{\beta(p)+1}}$,
\item[(iv)] if $p \equiv 3 \pmod 4$, $Y<p \le Z$ and
             $\pi_{\ell+2}(p) \not \in B$,\\ then $4y \equiv -1 \pmod
             {p^{\beta(p)+1}}$,
\item[(v)] if $p \equiv 3 \pmod 4$, $Z<p \le 4k$ and
             $\pi_{\ell+2}(p) \not \in A$,\\ then $4y \equiv -1 \pmod
             {p^{\beta(p)+1}}$.
\end{itemize}
We claim that
none of the numbers $y+1, \ldots, y+k$ is a sum of two squares,
which together with $y \le P$ and
\eqref{lowball} proves the theorem. To settle the
claim, let $1 \le j \le k$. If $\pi_\ell(j) \in S_\ell$, then by
Lemma \ref{lem1} and property (i) above $y+j$ cannot be a sum
of two squares, so we can assume that $\pi_\ell(j) \not \in S_\ell$, whence
$\pi_{\ell+2}(\tau(j)) \not \in A$.
Now $\tau(j) \equiv 3 \pmod 4$, so $\tau(j)$ must be
divisible by a prime $p$ with $p \equiv 3 \pmod 4$ where
$3 \le p \le 4k-1$ and $p^\gamma \mid \mid \tau(j)$ for odd $\gamma
\le \beta(p)$.

\smallskip
\noindent
\textbf{Case I:} $p \le X$. Then by property (ii) above $4y \equiv -1
\pmod {p^{\beta(p)+1}}$.

\smallskip
\noindent
\textbf{Case II:} $Z<p \le 4k$.
Then $p \equiv 3 \pmod 4$ and $p \mid \tau(j)$
imply that $\tau(j)=p$, so $\pi_{\ell+2}(\tau(j))=
\pi_{\ell+2}(p) \not \in A$
and by property (v) above, we have
$4y \equiv -1 \pmod {p^{\beta(p)+1}}$.

\smallskip
\noindent
\textbf{Case III:} $Y<p \le Z$.
Then $p \equiv 3 \pmod 4$ and $p \mid \tau(j)$ imply that
$\tau(j)=p$ or $\tau(j)=5p$. As above, if $\tau(j)=p$
we conclude that
$\pi_{\ell+2}(p) \not \in A$, whereas if $\tau(j)=5p$  we obtain
$\pi_{\ell+2}(5p) \not \in A$. Writing $p=\tau(s)$ for some
positive integer $s$, we then find that $\pi_\ell(s) \not \in V_\ell$,
whence by Lemma \ref{hoch_pokern} also $\pi_\ell(s) \not \in U_\ell$,
hence $\pi_{\ell+2}(p) \not \in B$. As $B \subset A$, we get
$\pi_{\ell+2}(p) \not \in B$ regardless of whether
$\tau(j)=p$ or $\tau(j)=5p$, so by property (iv) again
$4y \equiv -1 \pmod {p^{\beta(p)+1}}$.

\smallskip
\noindent
\textbf{Case IV:} $X<p \le Y$.
Then $p \equiv 3 \pmod 4$ and $p \mid \tau(j)$ imply
that $\tau(j)=p$ or $\tau(j)=5p$ or $\tau(j)=9p$.
If $\tau(j)=p$, then $\pi_{\ell+2}(p)=\pi_{\ell+2}(\tau(j)) \not \in A$.
Next, if $\tau(j)=5p$, then
$\pi_{\ell+2}(\tau(j))=\pi_{\ell+2}(5p) \not \in A$,
hence as above $\pi_{\ell+2}(p) \not \in B$ by Lemma \ref{hoch_pokern}.
Finally, if $\tau(j)=9p$, then
$\pi_{\ell+2}(\tau(j)) = \pi_{\ell+2}(9p) \not \in A$, hence as above
$\pi_{\ell+2}(p) \not \in C$ by Lemma \ref{qd}.
As $C \subset B \subset A$, we obtain $\pi_{\ell+2}(p) \not \in C$
regardless whether $\tau(j)=p$, $\tau(j)=5p$ or $\tau(j)=9p$,
whence  $4y \equiv -1 \pmod {p^{\beta(p)+1}}$ by property (iii) above.

\smallskip
In all cases,
\[
  4(y+j) \equiv 4j-1 \pmod {p^{\beta(p)+1}},
\]
so $p^\gamma \mid \mid (y+j)$. Since
$p \equiv 3 \pmod 4$ and $\gamma$ is odd, $y+j$ cannot be a sum of two
squares.

\bigskip
\begin{remark} One might wonder if the study of further iterations
leading to analogous sets
$D, E, \ldots$ would reduce the size of $P$ even more.
As far as we see this is not the case because $C \cap 13C = \emptyset$, but
this does not exclude other refinements.
\end{remark}
\section{Longer gaps between numbers representable by binary quadratic
forms of discriminant $D$}
Again, we follow the idea of Richards \cite{Richards:1982}:
Choose a positive integer $r \in \{1, \ldots, |D|\}$
such that the Kronecker symbol $(D/r)$ has value
\[
  \left( \frac{D}{r} \right)=-1.
\]
The following two well known results are provided for easy later reference.
\begin{lemma}
\label{hitze}
For fixed $m \in \mathbb{Z} \backslash\{0\}$
with $m \equiv 0 \pmod 4$ or $m \equiv 1
\pmod 4$, the Kronecker symbol $(\frac{m}{\cdot})$ is periodic of
period dividing $|m|$, i.e. for all $k, n \in \mathbb{Z}$ where
$n \ne 0$ and $n+km \ne 0$ we have
\[
  \left( \frac{m}{n+km} \right) = \left( \frac{m}{n} \right).
\]
\end{lemma}
\begin{proof}
This is Theorem 2.29 in \cite{Cohen}.
\end{proof}

\begin{lemma}
\label{welle}
Let $D$ be a fundamental discriminant, $n \in \mathbb{N}$ and $p$ be
a prime such that $p^\alpha$ exactly divides $n$ for odd $\alpha$ and
with  $(\frac{D}{p})=-1$. Then $n$ is not representable by
\emph{any} binary
quadratic form of discriminant $D$.
\end{lemma}
\begin{proof}
Let $R_D(n)$ be the total number of representations of $n$ by any binary
quadratic form of discriminant $D$. Then by Theorem 3 in
\cite[\S 8]{Zagier}, we have
\begin{equation}
\label{artie_shaw}
  R_D(n) = \sum_{\ell \mid n} \left( \frac{D}{\ell} \right).
\end{equation}
Since $(\frac{D}{\ell})$ is multiplicative in $\ell$, $R_D(n)$ is
multiplicative in $n$ as well. Now $\alpha$ is odd and $(\frac{D}{p})=-1$,
whence $R_D(p^\alpha)=0$. By multiplicativity, as
$p^\alpha$ exactly divides $n$, also $R_D(n)=0$, so indeed
$n$ is not represented by any binary quadratic form of discriminant $D$.
\end{proof}

Next, define the sequence $\ell_i$ recursively by
\begin{align*}
  \ell_1 & = 1,\\
  \ell_i & =\min\{ \ell \in \N: (\ell, D)=1 \; \text{and} \;
   \ell>\ell_{i-1}\} \quad (i>1).
\end{align*}
Further, for $i \in \N$ define
\begin{equation}
\label{brunnen}
  T_i=\{x \in (\Z/|D|\Z)^*: \ell_j x \equiv r \pmod {|D|}
  \; \text{for some} \; j \le i\},
\end{equation}
let $\pi$ be the projection
\[
  \pi : \Z \rightarrow \Z/|D|\Z,
\]
and introduce the abbreviations
\[
  t = \varphi(|D|)
\]
and
\[
  L=|D|(k+1).
\]
Moreover, fix $\varepsilon>0$ and in terms of $\varepsilon$ and $|D|$
fix a sufficiently large positive integer $k$, and for prime $p$ let
\[
  \beta(p) = \max_{p^m \le L} m.
\]
Finally, define
\begin{equation}
\label{freibad}
  P = \prod_{\substack{p_t \le L/\ell_t:\\(p_t, D)=1}}
      p_t^{\beta(p)+1} \quad
      \prod_{i=1}^{t-1} \quad
      \prod_{\substack{L/\ell_{i+1} < p_i \le L/\ell_i:\\
                          \pi(p_i) \in T_i}} p_i^{\beta(p)+1},
\end{equation}
where $p_1, \ldots, p_t$ denote prime numbers. Now
\[
  T_i \subset T_{i+1} \quad \text{and} \quad l_i \ge i \quad (i \in \N),
\]
hence
\begin{align*}
  \{p \le L/\ell_t: (p, D)=1\} & \subset \{p \le L/t:(p,D)=1\},\\
  \{L/\ell_t<p \le L/\ell_{t-1}:\pi(p) \in T_{t-1}\}
  & \subset \{p \le L/t:(p,D)=1\} \\
  & \quad\ \cup \{L/t<p \le L/(t-1):\pi(p) \in T_{t-1}\},\\
  \{L/\ell_{t-1}<p \le L/\ell_{t-2}:\pi(p) \in T_{t-2}\}
  & \subset 
  \{L/(t-1)<p \le L/(t-2):\pi(p) \in T_{t-2}\}\\
  & \quad \ \cup \{L/t<p \le L/(t-1):\pi(p) \in T_{t-1}\}\\
  & \quad \ \cup \{p \le L/t: (p, D)=1\}, 
\end{align*}
and so on. Therefore $P$ can be bounded above by
\[
  \prod_{\substack{p \le L/t:\\(p, D)=1}} p^{\beta(p)+1}
      \prod_{i=1}^{t-1}
         \prod_{\substack{L/(i+1) < p \le L/i:\\
                          \pi(p) \in T_i}} p^{\beta(p)+1}.
\]
 Using the observations
\[
  p^{\beta(p)+1} \le L^2
\]
and
\[
  \#T_i = i \quad (i \le t)
\]
together with the prime number theorem in arithmetic progressions,
we obtain
\[
  P \le L^{2(1+\varepsilon)\alpha}
\]
where
\begin{align*}
  \alpha & = \frac{L/t}{\log L/t} + \frac{L}{t}
    \sum_{i=1}^{t-1} \frac{i \left(1/i-1/(i+1)\right)}
      {\log(L/i-L/(i+1))}\\
  & = \frac{L/t}{\log L/t} + \frac{L}{t}
    \sum_{i=1}^{t-1} \frac{1/(i+1)}
      {\log (L/(i(i+1)))}\\
  & \le \frac{L/t}{\log L/t} + \frac{L}{t}
    \frac{1}{\log (L/(t(t+1))} \sum_{i=1}^{t-1} \frac{1}{i+1}\\
  & \le \frac{L/t}{\log L/t} + \frac{L}{t}
    \frac{\log t}{\log (L/(t(t+1)))}\\
  & \le \frac{L/t}{\log (L/(t(t+1)))} (1+\log t).\\
\end{align*}
With
\[
  \lim_{k \rightarrow \infty} \frac{\log L}{\log (L/(t(t+1)))} = 1,
\]
we obtain, for $k$ sufficiently large in terms of $|D|$ and 
$\varepsilon$,
\[
  P \le \exp \left( 2 (1+2\varepsilon) \frac{L}{t}(1+\log t) \right),
\]
and from this
\begin{equation}
\label{sonne}
  \frac{k+1}{\log P} \ge \frac{\varphi(|D|)}
  {2(1+2\varepsilon)|D|(1+\log \varphi(|D|)}.
\end{equation}
Now choose $y \in \{1, \ldots, P\}$ such that
\[
  |D|y \equiv r \pmod P
\]
which is possible as $(D,P)=1$ by definition \eqref{freibad} of $P$.
We claim that none of the numbers $y+1, \ldots, y+k$ can be represented
by any binary quadratic form of discriminant $D$, which together with
$y \le P$ and
\eqref{sonne} proves the theorem. To settle the claim, fix
$j \in \{1, \ldots, k\}$. Now
\begin{equation}
\label{serious}
  |D|(y+j) \equiv |D|j+r \pmod P.
\end{equation}
Since
\[
  -1=\left(\frac{D}{r}\right)=\left(\frac{D}{|D|j+r}\right)
\]
by Lemma \ref{hitze}, we conclude that $|D|j+r$ must be divisible by a
prime $p$ with $(D/p)=-1$ to an odd power $\gamma$ at most $\beta(p)$:
Writing
\[
  |D|j+r=p^\gamma \ell
\]
where $\ell$ is a certain positive integer
coprime to $D$ and $p$, we find that $|D|j+r \le |D|(k+1)=L$, so
$\gamma \le \beta(p)$.
If $\gamma \ge 3$, then $p \le L^{1/3}$.
As $L/\ell_t \ge L^{1/3}$ for sufficiently large $k$ (in terms of $D$),
by \eqref{freibad}
we can then assume that $p^{\beta(p)+1}$ divides $P$.
If $\gamma=1$, then $|D|j+r=p\ell$, so $p \le L$.
Moreover, if $L/\ell_{i+1}<p\le L/\ell_i$, then $\ell \le \ell_i$,
so $\pi(p) \in T_i$ by definition \eqref{brunnen} of $T_i$, whence
$p^{\beta(p)+1}$ again divides $P$ by \eqref{freibad}
as well as in case $p \le L/\ell_t$,
once more by definition \eqref{freibad}.
Using \eqref{serious}, we conclude that $p$ divides $y+j$ to
an odd power, so as $(D/p)=-1$ by Lemma \ref{welle} the number
$y+j$ indeed cannot be represented by any binary
quadratic form of discriminant $D$.

\section{Long gaps between sums of two squares in sparse sequences}
In order to prove Theorem \ref{shift}, we need the following auxiliary
result.
\begin{lemma}
\label{32grad}
Let $d, d'$ and $r$ be as in the statement of Theorem \ref{shift}, and let
$j$ be a positive integer. Then
\[ \gcd \left( \frac{(4d'j)^d-1}{4d'j-1},4d'j-1 \right)=1. \]
\end{lemma}
\begin{proof}
Let $a=4d'j$. Then the claim is
\[
  \gcd \left( \frac{(a^{d'})^{2^r}-1}{a-1},a-1 \right)=1.
\]
We prove this in two parts, first
\[
  \gcd\left( \frac{(a^{d'})^{2^r}-1}{a^{d'}-1},a-1\right)=1,
\]
and then
\[
  \gcd\left( \frac{a^{d'}-1}{a-1},a-1\right)=1.
\]
\textbf{Part 1:} Observe that
\[
  \gcd \left(\frac{(a^{d'})^{2^r}-1}{a^{d'}-1},a-1\right)
  \leq \gcd\left( \frac{(a^{d'})^{2^r}-1}{a^{d'}-1},a^{d'}-1\right),
\]
and therefore it is enough to prove that
\[
  \gcd \left( \frac{(a^{d'})^{2^r}-1}{a^{d'}-1},a^{d'}-1\right)=1.
\]
Now
\[ \frac{(a^{d'})^{2^r}-1}{a^{d'}-1}=\sum_{i=0}^{2^r-1}a^{d'i}=\prod_{u=0}^{r-1}(a^{2^ud'}+1), \]
as every integer $i$ can be written as a sum of powers of $2$.
Each factor $a^{2^ud'}+1$ is coprime to $a^{d'}-1$, as the integer $a^{2^ud'}-1$ is a multiple of $a^{d'}-1$,
at distance $2$ to $a^{2^ud'}+1$, and recall that $a^{d'}$ is even.

\smallskip
\noindent
\textbf{Part 2:}
We have the identity
\[ \begin{array}{rlll}
   d'+\sum_{i=0}^{d'-1} i x^{d'-1-i}(x-1) &=
   d'+x^{d'-1}+&2x^{d'-2}
+\ldots+(d'-1)x\\
   &&-x^{d'-2}
-\ldots-(d'-2)x-(d'-1)\\
   &=\sum_{i=0}^{d'-1} x^i.&
\end{array}\]
By this identity, applied with $a=4d'j=x$, a division of
$\frac{a^{d'}-1}{a-1}=\sum_{i=0}^{d'-1} a^i$ by $(a-1)$ leaves the remainder $d'$.
As $\gcd(4d'j-1,d')=1$ the claim follows.
\end{proof}
\begin{proof}[Proof of Theorem \ref{shift}]
It is enough to show the following:
For fixed $\varepsilon>0$, for all sufficiently large $k$ there
exists a positive integer
$y=y_k \le \exp ((1+\varepsilon)4kd')$ such that
none of the integers $y+f_d(j), 1\leq j \leq k$ is a sum of two squares.
Fix the gap size $k$.
For each prime $p \leq 4kd', p \equiv
3 \pmod 4$ let $\beta=\beta(p)$ be the highest power with
$p^{\beta}\leq 4kd'$. Let
\[
  P=\prod_{p \leq 4kd',\, p \equiv 3 \pmod 4} p^{\beta(p)+1}.
\]
Define $y \in \{1, \ldots, P\}$ by $(4d')^dy\equiv -1 \pmod P$.
Then none of the integers in
\[
  I=\{y+f_d(1), \ldots ,y+f_d(k)\}
\]
is the sum of two squares.
By the prime number theorem in arithmetic progressions,
\[
  P<(4kd')^{2 (1+\varepsilon)(2kd'/\log 4kd')}=\exp ((1+\varepsilon)4kd'),
\]
whence
\[
  k>\frac{\log P}{4d'(1+\varepsilon)}.
\]
Since $(4d')^dy \equiv -1 \pmod P$, we have that
\[(4d')^d(y+f_d(j))\equiv f_d(4d'j)-1 \pmod P, \text{ for } 1\leq j \leq k.\]
Now $\delta=4d'j-1$ is a divisor of $f_d(4d'j)-1$ and therefore
exactly divisible by some prime power
$p^{\alpha}$ with $p \equiv 3 \pmod 4$ and $\alpha$ odd.
By lemma \eqref{32grad}
the codivisor $\frac{(4d'j)^d-1}{4d'j-1}$ is coprime to $\delta$,
which implies that $f_d(4d'j)-1$ is exactly divisible by this
prime power $p^{\alpha}$.
As $\alpha \leq \beta(p)$, and $P$ is divisible by $p^{\beta+1}$, it follows
that $y+f_d(j)$ is also exactly divisible by $p^{\alpha}$, and is therefore 
not the sum of two squares.
\end{proof}

\bibliography{largegaps}
\bibliographystyle{plain}

\appendix
\section{Table of abbreviations}
\label{smalltalk}
In this appendix we briefly collect some notation used in the proof
of Theorem~\ref{staines} in section \ref{regen}. First recall the
maps
\[
  \pi_\ell : \Z \rightarrow \{1, \ldots, 2^\ell\}, \quad
  x \mapsto x \pmod {2^\ell}
\]
and
\[
  \tau : \Z \rightarrow \Z; \quad j \mapsto 4j-1.
\]
The sets $S_\ell$, $T_\ell$, $U_\ell$, $V_\ell$, $W_\ell$ and
$R_\ell$ are then defined by
\begin{align*}
  S_l & = \{2^ab \in \{1, \ldots, 2^\ell\}:
  a \le \ell-2, b \equiv 3 \pmod 4\} \quad (\ell \ge 2),\\
  T_{\ell+2} & = \pi_{\ell+2}(\tau(S_\ell)) \quad (\ell \ge 2),\\
  U_3 &= \{3\},\\
  U_\ell & = U_{\ell-1} \cup \{u+2^{\ell-1}: u \in U_{\ell-1}\}
  \cup \{3 \times 2^{\ell-2}\} \quad (\ell \ge 4),\\
  V_\ell &= \{s \in S_\ell: \pi_{\ell+2}(5\tau(s)) \in T_{\ell+2}\}
  \quad (\ell \ge 2),\\
  W_5 &=\{24\},\\
  W_\ell &= W_{\ell-1} \cup \{u+2^{\ell-1}: u \in W_{\ell-1}\}
  \cup \{3 \times 2^{\ell-2}\} \quad (\ell \ge 6),\\
  R_\ell &= \{s \in S_\ell: \pi_{\ell+2}(5\tau(s)) \in T_{\ell+2}
  \text{ and } \pi_{\ell+2}(9\tau(s)) \in T_{\ell+2}\} \quad
  (\ell \ge 2).
\end{align*}
Note the inclusions
\begin{align*}
  R_\ell & \subset V_\ell \subset S_\ell \quad (\ell \ge 5),\\
  W_\ell & \subset U_\ell \subset S_\ell \quad (\ell \ge 5),\\
  U_\ell & \subset V_\ell \quad (\ell \ge 3),\\
  W_\ell & \subset R_\ell \quad (\ell \ge 5).
\end{align*}
\section{Some examples for the sets of residue classes introduced
in section \ref{regen}}
\label{english}
We have
\begin{align*}
S_2&=\{3\},\\
S_3&=\{3,6,7\},\\
S_4&=\{3,6,7,11,12,14,15\},\\
S_5&=\{3, 6, 7, 11, 12, 14, 15, 19, 22, 23, 24, 27, 28, 30, 31\},\\
T_4&=\{11\},\\
T_5&=\{11,23,27\},\\
T_6&=\{11,23,27,43,47,55,59\},\\
T_7&=\{11, 23, 27, 43, 47, 55, 59, 75, 87, 91, 95, 107, 111, 119, 123 \},\\
U_3&=\{3\}=V_3 \subset S_3,\\
U_4&=\{3,11,12\}=V_4 \subset S_4,\\
U_5&=\{3,11,12,19,24,27,28\}=V_5 \subset S_5,\\
\pi_5(\tau(U_3))&=\{11\} \subset T_5,\\
\pi_6(\tau(U_4))&=\{11, 43, 47\} \subset T_6,\\
\pi_7(\tau(U_5))&=\{11, 43, 47, 75, 95, 107, 111\} \subset T_7,\\
W_5&=\{24\}=R_5,\\
\pi_7(\tau(W_5))&=\{95\} \subset \pi_7(\tau(U_5))
\subset T_7.
\end{align*}
\end{document}